\documentclass[12pt]{amsart}
\usepackage[margin=1in]{geometry}
\usepackage{amsmath}
\usepackage{amsthm}
\usepackage{amsbsy}
\usepackage{amssymb}
\usepackage{hyperref}
\usepackage{amsfonts}
\usepackage{color}
\usepackage{constants}

\usepackage{enumitem}
\numberwithin{equation}{section}
\newtheorem{theorem}{Theorem}[section]
\newtheorem{hypothesis}{Hypothesis}
\newtheorem{lemma}[theorem]{Lemma}
\newtheorem{proposition}[theorem]{Proposition}
\newtheorem{corollary}[theorem]{Corollary}

\newtheorem*{hypothesis*}{The Ramanujan-Petersson conjecture}

\theoremstyle{remark}

\theoremstyle{remark}

\newcommand{\Z}{\mathbb{Z}}
\newcommand{\R}{\mathbb{R}}
\newcommand{\Q}{\mathbb{Q}}
\newcommand{\F}{\mathbb{F}}
\newcommand{\N}{\mathbb{N}}
\newcommand{\p}{\mathfrak{p}}

\newcommand{\m}{\mathfrak{m}}

\newcommand{\f}{\mathfrak{f}}
\renewcommand{\m}{\mathbf{m}}

\newcommand{\re}{\textup{Re}}

\renewcommand{\a}{\mathfrak{a}}

\renewcommand{\d}{\mathfrak{d}}
\newcommand{\q}{\mathfrak{q}}

\makeatletter
\def\imod#1{\allowbreak\mkern10mu({\operator@font mod}\,\,#1)}
\makeatother

\renewcommand{\pmod}[1]{\left(\mathrm{mod}\,\,#1\right)}

\newconstantfamily{abcon}{symbol=c}

\title[Bounded Gaps Between Primes in Hecke Equidistribution Problems]
{Bounded Gaps Between Primes in Multidimensional Hecke Equidistribution Problems}
\author{Jesse Thorner}
\date{\today}
\begin{document}

\begin{abstract}
Using Duke's large sieve inequality for Hecke Gr{\"o}ssencharaktere and the new sieve methods of Maynard and Tao, we prove a general result on gaps between primes in the context of multidimensional Hecke equidistribution.  As an application, for any fixed $0<\epsilon<\frac{1}{2}$, we prove the existence of infinitely many bounded gaps between primes of the form $p=a^2+b^2$ such that $|a|<\epsilon\sqrt{p}$.  Furthermore, for certain diagonal curves $\mathcal{C}:ax^{\alpha}+by^{\beta}=c$, we obtain infinitely many bounded gaps between the primes $p$ such that $|p+1-\#\mathcal{C}(\F_p)|<\epsilon\sqrt{p}$.
\end{abstract}

\maketitle

\section{Introduction and statement of results}

Conjectures about primes represented by polynomials of degree greater than one have been considered by number theorists for well over a century.  It is conjectured that every irreducible polynomial of degree at least one in $\Z[x]$ represents infinitely many primes, but this is known unconditionally for only the degree one case by Dirichlet's work in 1837.  The simplest degree two polynomial to study is $x^2+1$.  A notable partial result due to Iwaniec \cite{Iwaniec} states that there are infinitely many $n$ such that $n^2+1$ is a product of at most two primes.  By the work of Lemke Oliver \cite{RJLO}, the same can be said for any irreducible polynomial $f(x)$ of degree two such that $f(x)\not\equiv x^2+x\pmod 2$.

By extending the question to consider primes represented by multivariate polynomials, one can prove much stronger results.  For example, any positive definite binary quadratic form $ax^2+bxy+cy^2$ of discriminant $-D$ represents a positive proportion of the primes, where the proportion depends on the class number of binary quadratic forms of discriminant $-D$.  Using highly technical sieve methods, Friedlander and Iwaniec  \cite{FI} proved an asymptotic formula for the number of primes of the form $x^2+y^4$, and Heath-Brown \cite{HB3} did the same for primes of the form $x^3+2y^3$.

As an approximation to understanding the distribution of primes of the form $n^2+1$, one might ask for the distribution of primes $p=a^2+b^2$ where is small in terms of $p$.  One notices that if $p=a^2+b^2$ where $a,b\in\Z$ and $a$ is odd, then $p$ splits completely in $\Q(\sqrt{-1})$, in which case we have that $p=(a+bi)(a-bi)$.  We then have that $\cos(\arg(a+bi))=\frac{a}{\sqrt{p}}$, which leads to the study of the distribution of $\cos(\arg(a+bi))\in[-1,1]$.  A classical result of Hecke states that the values $\frac{a}{\sqrt{p}}$ are equidistributed in $[\alpha,\beta]$ with respect to the measure $\frac{1}{\pi}\frac{dt}{\sqrt{1-t^2}}$; that is, if $[\alpha,\beta]\subset[-1,1]$ is a fixed subinterval and $\pi(x):=\#\{p\leq x\}$, then
\begin{equation}
\label{eqn:1}
\lim_{x\to\infty}\frac{1}{\pi(x)}\#\left\{p\leq x:p=a^2+b^2,\frac{a}{\sqrt{p}}\in [\alpha,\beta]\right\}=\frac{1}{2\pi}\int_\alpha^\beta\frac{1}{\sqrt{1-t^2}}dt.
\end{equation}
This equidistribution law is equivalent to the statement that $L$-functions associated to Hecke Gr{\"o}ssencharactere (henceforth referred to as Hecke characters) for the field $\Q(\sqrt{-1})$ have no zeros on the line $\re(s)=1$. This bears resemblance with proof of the prime number theorem for arithmetic progressions $a\bmod q$, with the role of a residue class modulo $q$ replaced with the role of a subinterval of $[-1,1]$ and the role of Dirichlet $L$-functions replaced with the role of Hecke $L$-functions.  By taking $[\alpha,\beta]$ to be a small interval centered at $0$, Ankeny \cite{Ankeny} and Kubilius \cite{Kubilius} used the Generalized Riemann Hypothesis to prove that there are infinitely many primes $p=a^2+b^2$ with $a=O(\log p)$.  Unconditionally, Kubilius \cite{Kubilius} proved that there are infinitely many primes $p=a^2+b^2$ with $a=O(p^{25/64})$.  The current record is due to Harman and Lewis \cite{HL}, who proved using sieve methods that we may take $a=O(p^{\theta})$ with $\theta<0.119$.

More generally, we let $K$ be a number field, and let $\mathrm{N}=\mathrm{N}_{K/\Q}$ denote the absolute field norm of $K$.  Duke \cite{Duke} studied a generalization of Ankeny's work by replacing $a^2+b^2$ with an arbitrary norm form over $K$ given by
\[
f(\vec{x})=\mathrm{N}\left(\sum_{j=1}^{[K:\Q]}\alpha_j x_j\right)\mathrm{N}\a^{-1},\qquad \vec{x}=(x_1,\ldots, x_{[K:\Q]}),
\]
which is defined with respect to an ideal $\a$ with a special type of integral basis $\{\alpha_1,\ldots,\alpha_{[K:\Q]}\}$.  Instead of using the Generalized Riemann Hypothesis, Duke used a zero density estimate for Hecke $L$-functions to study the distribution of the primes in the set
\begin{equation}
\label{eqn:norm_primes}
\mathcal{P}_{f,\mathcal{I},K}=\left\{p\colon\textup{there exists $\vec{x}\in\Z^{[K:\Q]}$ such that $f(\vec{x})=p$ and $\frac{1}{p^{\frac{1}{[K:\Q]}}}\vec{x}\in\prod_{I\in\mathcal{I}}I$}\right\},
\end{equation}
where $\mathcal{I}=\{I_1,\ldots, I_{[K:\Q]}\}$ is a collection of subintervals of $[-1,1]$.  By Hecke, the primes in $\mathcal{P}_{f,\mathcal{I},K}$ satisfy an equidistribution law which generalizes \eqref{eqn:1}.  Using this equidistribution law, Duke proved that if $m\in\{1,\ldots,[K:\Q]\}$ and $0\leq\delta<\frac{1}{3[K:\Q]}$ are fixed, then
\begin{align*}
\#\{p\leq x:\textup{there exists $\vec{x}\in\Z^{[K:\Q]}$ }&\textup{such that $p=f(\vec{x})$}\notag\\
&\textup{and $|x_j|\leq p^{\frac{1}{[K:\Q]}-\delta}$ for all $j\neq m$}\}\asymp \frac{x^{1-([K:\Q]-1)\delta}}{\log x}.
\end{align*}

In addition to studying the distribution of primes represented by a single multivariate form, one can ask questions about the distribution of primes represented simultaneously by several univariate linear forms $n+h_i$, where $1\leq i\leq k$.  Setting $\mathcal{H}_k=\{h_1,\ldots,h_k\}$, we call $\mathcal{H}_k$ an admissible set if for all primes $p$ there exists an integer $n_p$ such that $\prod_{i=1}^k (n_p+h_i)$ and $p$ are coprime.  The {\it prime $k$-tuples conjecture}, first conjectured by Hardy and Littlewood, asserts that if $\mathcal{H}_k$ is admissible, then there exists a positive constant $\mathfrak{S}=\mathfrak{S}(\mathcal{H}_k)$ such that as $x\to\infty$,
\begin{equation}
\label{eqn:HL}
\#\{n\leq x : \#(\{n+h_1,\ldots,n+h_k\}\cap\mathbb{P})=k\}\sim\mathfrak{S}\frac{x}{(\log x)^k},
\end{equation}
where $\mathbb{P}$ denotes the set of all primes.  When $k=2$ and $\mathcal{H}_2=\{0,2\}$, this problem reduces to the twin prime conjecture.

The prime $k$-tuples conjecture is completely open for $k>1$, but the last decade has seen many strong approximations to the conjecture.  The first such approximation was proven by Goldston, Pintz, and Y{\i}ld{\i}r{\i}m \cite{GPY}; they proved that
\[
\liminf_{n\to\infty}\frac{p_{n+1}-p_n}{\log p_n}=0,
\]
where $p_n$ is the $n$-th prime.  This is quite remarkable, considering that the average size of $p_{n+1}-p_n$ is roughly $\log p_n$ by the prime number theorem.  Their work was substantially improved upon by Zhang \cite{Zhang}, who proved for the very first time that there exist infinitely many bounded gaps between primes; specifically,
\[
\liminf_{n\to\infty}(p_{n+1}-p_n)<7\times10^{7}.
\]
Using an approach very different from that of Zhang, Maynard \cite{maynard} proved that
\[
\liminf_{n\to\infty}(p_{n+1}-p_n)\leq 600.
\]
Furthermore, for any $m\geq1$, Maynard's work yields the bound
\[
\liminf_{n\to\infty}(p_{m+n}-p_n)\ll m^3 e^{4m}.
\]
(The underlying improvement to the Selberg sieve which lead to this result was independently by Tao, who arrived at slightly different conclusions.)  These qualitative approximations of \eqref{eqn:HL} stem from the result that if $m\geq2$, then there exists a constant $k_0(m)$ such that for any admissible $k$-tuple $\mathcal{H}_k=\{h_1,\ldots,h_k\}$ with $k\geq k_0(m)$, then there are infinitely many $n$ such that at least $m$ of the $n+h_1,\ldots,n+h_k$ are simultaneously prime.  For example, Maynard's bound of 600 follows from showing that one may take $k_0(2)=105$ and choosing $\mathcal{H}_{105}$ appropriately.  In more recent work \cite{maynard2}, Maynard proved a quantitative approximation of \eqref{eqn:HL}; in particular, he proved that there exists an absolute constant $C\geq1$ such that if $k>C$ and $\mathcal{H}_k=\{h_1,\ldots,h_k\}$ is admissible, then
\begin{equation}
\label{eqn:k-tuples_0}
\#\{n\leq x : \#(\{n+h_1,\ldots,n+h_k\}\cap\mathbb{P})\geq C^{-1}\log k\}\gg\frac{x}{(\log x)^k}.
\end{equation}

The author extended the work on bounded gaps between primes to the context of the Chebotarev density theorem.  Specifically, let $K/\Q$ be a Galois extension of number fields with Galois group $G$ and absolute discriminant $\Delta$, and let $C\subset G$ be a conjugacy class.  Define
\[
\mathcal{P}_C=\left\{p\nmid\Delta:\left[\frac{K/\Q}{p}\right]=C\right\},
\]
where $[\frac{K/\Q}{p}]$ denotes the Artin symbol, and let $\varphi(q)$ be Euler's totient function.  In \cite{JT}, the author used Maynard's methods to prove that there are infinitely many positive integers $N$ such that for some $n\in[N,2N]$, we have that
\begin{equation}
\label{eqn:k-tuples_CDT}
\#(\{n+h_1,\ldots,n+h_k\}\cap\mathcal{P})\geq\left(\frac{1}{2}\min\left\{\frac{1}{2},\frac{2}{|G|}\right\}\frac{|C|}{|G|}\frac{\varphi(\Delta)}{\Delta}+o_{k\to\infty}(1)\right)\log k.
\end{equation}
The author explored applications of this result to the distribution of ranks of quadratic twists of elliptic curves, congruence conditions for the Fourier coefficients of modular forms, and primes represented by binary quadratic forms.  A more quantitative version \eqref{eqn:k-tuples_CDT} similar to \eqref{eqn:k-tuples_0} can be found in Theorem 3.5 of \cite{maynard2}.

Our first result is a proof of the infinitude of bounded gaps between the primes in sets of the form $\mathcal{P}_{f,\mathcal{I},K}$ defined by \eqref{eqn:norm_primes}, extending the work of Maynard to the setting of Duke's work in \cite{Duke}.
\begin{theorem}
\label{mainthm:angles1}
Define $\mathcal{P}_{f,\mathcal{I},K}$ as in \eqref{eqn:norm_primes}.  There exists a constant $C_{f,\mathcal{I},K}>0$ such that if $k>C_{f,\mathcal{I},K}$ and $\mathcal{H}_k=\{h_1,\ldots,h_k\}$ is an admissible set, then
\[
\#\{n\leq x: \#(\{n+h_1,\ldots,n+h_k\}\cap\mathcal{P}_{f,\mathcal{I},K})\geq C_{f,\mathcal{I},K}^{-1}\log k\}\gg\frac{x}{(\log x)^k}.
\]
\end{theorem}

By choosing $K=\Q(\sqrt{-1})$ and choosing $\mathcal{I}$ to have one short interval centered at zero, we immediately obtain the following result on gaps between primes which are nearly a square.

\begin{corollary}
Fix $0<\epsilon<\frac{1}{2}$ and $m\geq 1$, and let
\[
\mathcal{P}_{\epsilon}=\{p:p=a^2+b^2,|a|\leq\epsilon\sqrt{p}\}.
\]
There exists a constant $C_{\epsilon}>0$ such that if $k>C_{\epsilon}$ and $\mathcal{H}_k=\{h_1,\ldots,h_k\}$ is an admissible set, then
\[
\#\{n\leq x: \#(\{n+h_1,\ldots,n+h_k\}\cap\mathcal{P}_{\epsilon})\geq C_{\epsilon}^{-1}\log k\}\gg\frac{x}{(\log x)^k}.
\]
\end{corollary}

A similar sort of equidistribution problem lies in counting the number of $\F_p$-rational points on a given curve $\mathcal{C}/\Q$, where $\F_p$ is the finite field of order $p$.  We consider the class of diagonal curves given by
\[
\mathcal{C}:aX^{\alpha}+bY^{\beta}=c,
\]
where $a,b,c,\alpha,\beta\in\Z-\{0\}$ and $\alpha\geq\beta\geq2$.  Let $d=\gcd(\alpha,\beta)$ and $M=\mathrm{lcm}(\alpha,\beta)$, and let
\[
g=\frac{(\alpha-1)(\beta-1)-(d-1)}{2}
\]
be the genus of $\mathcal{C}$.  Define the trace of Frobenius for $\mathcal{C}$ at $p$ by
\[
a_{\mathcal{C}}(p)=p+1-\mathcal{N}_d-\#\mathcal{C}(\F_p),\qquad \mathcal{N}_d=\begin{cases}
d&\mbox{if $-a/b$ is a $d$-th power modulo $p$,}\\
0&\mbox{otherwise}.
\end{cases}
\]
Hasse proved that for each $p\equiv1\pmod M$ with $p\nmid abc$, we have that $|a_{\mathcal{C}}(p)|\leq2g\sqrt{p}$.  It follows from the work of Hecke that the sequence $\{a_{\mathcal{C}}(p)/(2g\sqrt{p})\}$ is equidistributed in $[-1,1]$ with respect to a certain probability measure, which depends on the curve.  When $g=1$, in which case $\mathcal{C}$ is an elliptic curve over $\Q$ with complex multiplication, this measure is $\frac{1}{2\pi}\frac{dt}{\sqrt{1-t^2}}+\frac{1}{2}\delta_{0}$, where $\delta_0$ is the Dirac delta function centered at zero.  When $g=2$, there are at most 52 possible probability measures; the particular measure is dictated by the Sato-Tate group of $\mathcal{C}$, as shown by Fit\'e, Kedlaya, Rotger, and Sutherland \cite{Kedlaya-Sutherland}.  Duke \cite{Duke} used this equidistribution law to study sets of the form
\begin{equation}
\label{eqn:curve_trace_set}
\mathcal{P}_{\mathcal{C},I}=\left\{p:\textup{$p\equiv 1\pmod{M},~p\nmid abc,~$ and $~\frac{p+1-\mathcal{N}_d-\#\mathcal{C}(\F_p)}{2g\sqrt{p}}\in I$}\right\},
\end{equation}
where $I$ is a subinterval of $[-1,1]$.  By considering $I$ to be a short interval centered at zero, Duke proved that for any fixed $0\leq\delta<(3\varphi(M))^{-1}$, we have
\begin{equation}
\label{eqn:duke2}
\#\{p\leq x:p\equiv 1\pmod M,~p\nmid abc,~|a_{\mathcal{C}}(p)|\leq 2gp^{1/2-\delta}\}\gg \frac{x^{1-\delta\varphi(M)/2}}{\log x}.
\end{equation}

Our second result is a proof of the infinitude of bounded gaps between primes in sets of the form \eqref{eqn:curve_trace_set}.

\begin{theorem}
\label{mainthm:angles2}
Let $\mathcal{C}:aX^\alpha+bY^\beta=c$ be a curve of genus $g$, where $a,b,c,\alpha,\beta\in\Z-\{0\}$ and $\alpha\geq\beta\geq2$.  Let $\mathcal{P}_{\mathcal{C},I}$ be defined by \eqref{eqn:curve_trace_set}.  There exists a constant $C_{\mathcal{C},I}>0$ such that if $k>C_{\mathcal{C},I}$ and $\mathcal{H}_k=\{h_1,\ldots,h_k\}$ is an admissible set, then
\[
\#\{n\leq x : \#(\{n+h_1,\ldots,n+h_k\}\cap\mathcal{P}_{\mathcal{C},I})\geq C_{\mathcal{C},I}^{-1}\log k\}\gg\frac{x}{(\log x)^k}.
\]
In particular, if $0<\epsilon<\frac{1}{2}$ and we define
\[
\mathcal{P}_{\mathcal{C},\epsilon}=\left\{p:\textup{$p\equiv 1\pmod{M},~p\nmid abc,~$ and $~|p+1-\mathcal{N}_d-\#\mathcal{C}(\F_p)|\leq\epsilon\sqrt{p}$}\right\},
\]
there exists a constant $C_{\mathcal{C},\epsilon}>0$ such that if $k>C_{\mathcal{C},\epsilon}$ and $\mathcal{H}_k=\{h_1,\ldots,h_k\}$ is an admissible set, then
\[
\#\{n\leq x : \#(\{n+h_1,\ldots,n+h_k\}\cap\mathcal{P}_{\mathcal{C},\epsilon})\geq C_{\mathcal{C},\epsilon}^{-1}\log k\}\gg\frac{x}{(\log x)^k}.
\]
\end{theorem}

As an example, we consider the genus 2 curve $\mathcal{C}:y^2=x^5+1$.  Defining
\[
a_{\mathcal{C}}(p):=p+1-\#\mathcal{C}(\F_p), 
\]
we have that if $p\equiv 1\pmod {5}$, then
\[
|a_{\mathcal{C}}(p)|\leq 4\sqrt{p}.
\]
We expect that for any fixed $t\equiv 6\pmod{10}$, there are infinitely many primes $p\equiv 1\pmod{5}$ such that $a_{\mathcal{C}}(p)=t$; this is reasonable to expect in light of the Lang-Trotter conjecture for elliptic curves \cite{LT}.  Assuming the Generalized Riemann Hypothesis for Hecke characters modulo 25 over the field $\Q(e^{2\pi i/5})$, Sarnak \cite{Sarnak} showed that there are $\gg \sqrt{x}$ primes $p\in[x,2x]$ such that $p\equiv 1\pmod 5$ and
\[
|a_{\mathcal{C}}(p)|\ll \log p.
\]
This provides an analogue of the aforementioned conditional results of Kubilius and Ankeny.

By Duke's inequality \eqref{eqn:duke2}, we have unconditionally that if $0<\delta<\frac{1}{12}$, then
\[
\#\left\{p\leq x:p\equiv 1\pmod 5, |a_{\mathcal{C}}(p)|\leq p^{1/2-\delta}\right\}\gg\frac{x^{1-2\delta}}{\log x}.
\]
The following result follows directly from Theorem \ref{mainthm:angles2}.

\begin{corollary}
Fix $0<\epsilon<\frac{1}{2}$.  Let $\mathcal{C}/\Q$ be the curve defined by $y^2=x^5+1$, and define
\[
\mathcal{P}_{\mathcal{C},\epsilon}=\left\{p:\textup{$p\equiv 1\pmod{5}~$ and $~|p+1-\#\mathcal{C}(\F_p)|\leq\epsilon\sqrt{p}$}\right\}.
\]
There exists a constant $C_{\mathcal{C},\epsilon}>0$ such that if $k>C_{\mathcal{C},\epsilon}$ and $\mathcal{H}_k=\{h_1,\ldots,h_k\}$ is an admissible set, then
\[
\#\{n\leq x : \#(\{n+h_1,\ldots,n+h_k\}\cap\mathcal{P}_{\mathcal{C},\epsilon})\geq C_{\mathcal{C},\epsilon}^{-1}\log k\}\gg\frac{x}{(\log x)^k}.
\]
\end{corollary}

In Section 2, we prove a zero density estimate for Hecke characters which generalizes the work of Montgomery \cite{Montgomery}.  The proof uses Duke's large sieve inequality for Hecke characters \cite{Duke}.  In Sections 3 and 4, we use the zero density estimate to prove a Bombieri-Vinogradov type estimate for primes satisfying an equidistribution law dictated by several independent Hecke characters.  In Section 5, we prove a general result on bounded gaps between primes satisfying a Hecke equidistribution condition using the Bombieri-Vinogradov type estimate a result of Maynard \cite{maynard2}; it is from this result that Theorems \ref{mainthm:angles1} and \ref{mainthm:angles2} will follow.

\subsection*{Acknowledgements}

The author thanks Michael Griffin, Robert Lemke Oliver, Ken Ono, and Frank Thorne for helpful discussions.

\section{A Zero Density Estimate for Hecke $L$-functions}

Let $K/\Q$ be a number field of degree $n_K:=[K:\Q]$.  Let $\q$ be an integral ideal, and let $\xi$ be a narrow class character modulo $\q$.  For a vector $\m=(m_1,m_2,\ldots,m_{n_K-1})\in\Z^{n_K-1}$, we define
\[
\lambda^{\m}=\prod_{j=1}^{n_K-1}\lambda_j^{m_j},
\]
where $\{\lambda_{1},\ldots,\lambda_{n_K-1}\}$ is a basis for the torsion-free Hecke characters modulo $\q$.  The implied constants in all asymptotic inequalities depend at most on $K$.

We begin by recalling the large sieve inequality for Hecke characters which was proven by Duke \cite[Theorem 1.1]{Duke}.
\begin{theorem}
\label{thm:duke_sieve}
If $c(\a)$ is a function on the ideals of $K$ and $\|c\|^2=\sum_{\mathrm{N}\a\leq N}|c(\a)|^2$, then
\[
\sum_{\mathrm{N}\q\leq Q}~\sideset{}{^*}\sum_{\xi\bmod\q}~\sum_{|\m|\leq T}\int_{-T}^{T}\left|\sum_{\mathrm{N}\a\leq N} c(\a)\xi\lambda^{\m}(\a)\mathrm{N}\a^{it}\right|^2 dt\ll_K (N+Q^2 T^{n_K})(\log QT)^A\|c\|^2,
\]
where $^*$ denotes summing over primitive characters, $A$ depends in an effectively computable manner on at most $K$, and $|\m|^2=\sum_{j=1}^{n_K-1}|m_j|^2$.
\end{theorem}

As a consequence of Theorem \ref{thm:duke_sieve}, we obtain the following fourth moment estimate.

\begin{theorem}
\label{thm:duke_moment}
We have
\[
\sum_{\mathrm{N}\q\leq Q}~\sideset{}{^*}\sum_{\xi\bmod\q}~\sum_{|\m|\leq T}\int_{-T}^{T}|L(1/2+it,\xi\lambda^{\m})|^4 dt\ll_K Q^2 T^{n_K}(\log QT)^A.
\]
\end{theorem}
\begin{proof}
The proof is essentially the same as \cite[Theorem 2.2]{Duke}.  By standard methods, the problem is reduced to proving that
\[
\sum_{\mathrm{N}\q\leq Q}~\sideset{}{^*}\sum_{\xi\bmod\q}~\sum_{|\m|\leq T}\int_{BT}^{2BT}|L(1/2+it,\xi\lambda^{\m})|^4 dt\ll_K Q^2 T^{n_K}(\log QT)^A,
\]
where $B$ is a large constant.  The approximate functional equation given by \cite[Theorem 2]{Hecke} tells us that if $BT\leq t\leq 2BT$ and $|\m|\leq T$, then
\begin{align*}
L(1/2+it,\xi\lambda^m)^2&=\sum_{\a}d(\a)\xi\lambda^{\m}(\a)\mathrm{N}\a^{-1/2-it}g_1(\mathrm{N}\a/x)\\
&+O\left(\sum_{\a}d(\a)\overline{\xi\lambda^{\m}}(\a)\mathrm{N}\a^{-1/2+it}g_2(\mathrm{N}\a/y)\right)+O(1),
\end{align*}
where $g_1$ and $g_2$ have compact support (see \cite{Hinz}) and
\[
xy=\left(|d|\mathrm{N}\f_{\xi}\left(\frac{|t|}{2\pi}\right)^{n_K}\right)^2.
\]
Thus $L(1/2+it,\xi\lambda^{\m})^2$ may be approximated by two finite Dirichlet series with approximately $T^{n_K}$ terms in the given range with bounded error.  Expressing $|L(1/2+it,\xi\lambda^{\m})|^4$ in terms of the functional equation, the left hand side of the theorem is bounded using Theorem \ref{thm:duke_sieve} as in \cite{Duke, Montgomery}.
\end{proof}

We now prove a zero density estimate similar to that of \cite[Theorem 2.1]{Duke} which allows us to average over primitive characters with modulus up to a given bound.

\begin{theorem}
\label{thm:duke_density}
Let $Q\gg1$, $T\gg1$,
\[
N(\sigma,T,\xi\lambda^{\m}):=\#\{\rho=\beta+i\gamma:L(\rho,\xi\lambda^{\m})=0,\beta\geq\sigma,|\gamma|\leq T\},
\]
and
\[
N(\sigma,Q,T):=\sum_{\mathrm{N}\q\leq Q}~\sideset{}{^*}\sum_{\xi\bmod\q}~\sum_{|\m|\leq T}N(\sigma,T,\xi\lambda^{\m}).
\]
Suppose that $n_K\geq2$.  There exists a constant $A>0$, depending in an effectively computable manner on at most $K$, such that for all $\sigma\in[0,1]$, we have that
\[
N(\sigma,Q,T)\ll (Q^2 T^{n_K})^{\frac{3(1-\sigma)}{2-\sigma}}(\log T)^A.
\]
\end{theorem}
\begin{proof}
The proof is essentially the same as \cite[Theorem 2.1]{Duke}.  Let
\[
M_x=M_x(s,\xi\lambda^{\m})=\sum_{\mathrm{N}\a\leq x}\mu(\a)\xi\lambda^{\m}(\a)\mathrm{N}\a^{-s},
\]
where $\mu$ is the M\"obius function for $K$, and let
\[
b(\a)=\sum_{\substack{\d\mid\a \\ \mathrm{N}\a\leq x}}\mu(\d).
\]
Then for $\re(s)>1$, we have that
\[
M_x(s)L(s,\xi\lambda^{\m})=1+\sum_{\mathrm{N}\a>x}b(\a)\xi\lambda^{\m}(\a)\mathrm{N}\a^{-s}.
\]
We now choose $y>0$, and we smooth to get
\begin{align}
\label{2.2.2}
&e^{-1/y}+\sum_{\mathrm{N}\a>x}b(\a)\xi\lambda^{\m}(\a)\mathrm{N}\a^{-s}e^{-\mathrm{N}\a/y}\\
=&M_x(s,\xi\lambda^{\m})L(s,\xi\lambda^{\m})+\frac{1}{2\pi i}\int_{(\frac{1}{2}-\sigma)}M_x(s+w,\xi\lambda^{\m})L(s+w,\xi\lambda^{\m})\Gamma(w)y^w dw,\notag
\end{align}
where $\sigma\in(1/2,1]$ and $\xi\lambda^{\m}$ is nontrivial.  If $\xi\lambda^{\m}$ is trivial, there is an additional error term whose effect is negligible due to the exponential decay of $\Gamma(2-s)$ in vertical strips.

Let $\rho=\beta+i\gamma$ be a zero of $L(s,\xi\lambda^{\m})$ with $\beta\geq\sigma$ and $|\gamma|\leq T$.  From \eqref{2.2.2}, we have that
\begin{align}
\label{2.2.3}
\left|\sum_{\mathrm{N}\a>x}b(\a)\xi\right.&\left.\lambda^{\m}(\a)\mathrm{N}\a^{-\rho}e^{-\mathrm{N}\a/y}\right|\\
&+y^{1/2-\sigma}\int_{\gamma-(\log T)^2}^{\gamma+(\log T)^2}|L(1/2+it,\xi\lambda^{\m})|\cdot|M_x(1/2+it,\xi\lambda^{\m})|dt\gg1.\notag
\end{align}
There are three possibilities for $\rho$:
\begin{enumerate}
\item $\left|\sum_{\mathrm{N}\a> x}b(\a)\xi\lambda^{\m}(\a)\mathrm{N}\a^{-\rho}e^{-\mathrm{N}\a/y}\right|\gg_K 1$.
\item For some $t_{\rho}$ such that $|t_{\rho}-\gamma|<(\log T)^2$, we have $|M_x(1/2+it_{\rho},\xi\lambda^{\m})|>x^{\sigma-1/2}$.
\item $\int_{\gamma-(\log T)^2}^{\gamma+(\log T)^2}|L(1/2+it,\xi\lambda^{\m})|dt\gg_K (y/x)^{\sigma-1/2}$.
\end{enumerate}
For $j=1,2,$ and $3$, let $N_j$ be the number of zeros $\rho$ satisfying conditions 1, 2, 3, respectively.  We choose a subset $R_j$ of zeros from each class for which the associated set of Hecke characters
\[
\Omega_j=\{\omega(\a)=\xi\lambda^{\m}(\a)\mathrm{N}\a^{-i\gamma}:|\m|\leq T,\textup{ $\xi\bmod\q$ primitive, }\mathrm{N}\q\leq Q,\rho\in R_j\}
\]
is $\gg (\log T)^2$ well spaced and is such that
\begin{equation}
\label{2.2.4}
N_j\ll |R_j|(\log T)^3,\qquad j=1,2,3.
\end{equation}

Since
\[
\sum_{j=1}^{3}|N_j|\ll (\log T)^3\sum_{j=1}^3 |R_j|,
\]
the theorem follows once we show that for $j=1,2,3,$
\[
|R_j|\ll (Q^2 T^{n_K})^{\frac{3(1-\sigma)}{2-\sigma}}.
\]

\subsection*{Case 1}
Let $\{I_k\}$ be a cover of $[x,y]$ by $\ll \log T$ intervals of the form $[N_k,2N_k]$.  By the Cauchy-Schwarz inequality, Theorem \ref{thm:duke_sieve}, and partial summation, we have that
\begin{align*}
\label{2.2.5}
|R_1|&\ll \sum_{\Omega_1}\left|\sum_{\a}b(\a)\omega(\a)\mathrm{N}\a^{-\beta}\right|^2\notag\\
&\ll \log T\sum_{k}\sum_{\Omega_1}\left|\sum_{\mathrm{N}\a\in I_k}b(\a)\omega(\a)\mathrm{N}\a^{-\beta}\right|^{2}\notag\\
&\ll \log T\sum_{k}(|I_k|+Q^2 T^{n_K}(\log QT)^A)|I_k|^{1-2\sigma}(\log QT)^A\notag\\
&\ll (y^{2-2\sigma}+Q^2 T^{n_K}x^{1-2\sigma})(\log T)^A.
\end{align*}
Setting $y=x^{3/2}$ and $x=(Q^2 T^{n_K})^{\frac{1}{2-\sigma}}$, we conclude that
\begin{equation}
\label{2.2.6}
|R_1|\ll (Q^2 T^{n_K})^{\frac{3(1-\sigma)}{2-\sigma}}(\log T)^A.
\end{equation}

\subsection*{Case 2}

By arguments similar to those in the previous case, we conclude that
\begin{equation}
\label{2.2.7}
|R_2|\ll (Q^2 T^{n_K})^{\frac{3(1-\sigma)}{2-\sigma}}(\log T)^A.
\end{equation}

\subsection*{Case 3}

Writing $l_{\rho}=(\gamma-(\log T)^2,\gamma+(\log T)^2)$, we use H\"older's inequality and Theorem \ref{thm:duke_moment} to obtain
\begin{align*}
|R_3|x^{2\sigma-1}&\ll \sum_{\rho\in R_3}\left|\int_{l_{\rho}}|L(1/2+it,\xi\lambda^{\m})|dt\right|^4\\
&\ll (\log T)^A \sum_{\rho\in R_3}\int_{l_{\rho}}|L(1/2+it,\xi\lambda^{\m})|^4 dt\\
&\ll (\log T)^A\sum_{\mathrm{N}\q\leq Q}\sum_{\xi\bmod\q}\sum_{|\m|\leq T}\int_{-2T}^{2T}|L(1/2+it,\xi\lambda^{\m})|^4 dt\\
&\ll Q^2 T^{n_K}(\log T)^A.
\end{align*}
Choosing $x$ as before, we see that
\[
|R_3|\ll (Q^2 T^{n_K})^{\frac{3(1-\sigma)}{2-\sigma}}(\log T)^A.
\]
\end{proof}
\section{A Bombieri-Vinogradov Estimate for Hecke characters}
We proceed to prove a mean value theorem of Bombieri-Vinogradov type for prime ideals of a number field $K/\Q$ which satisfy an equidistribution law dictated by Hecke characters.  The key input in the proof will be the zero density estimate in Theorem \ref{thm:duke_density}.  
Let $K/\Q$ be a number field of degree $n_K=[K:\Q]$ with ring of integers $\mathcal{O}_K$, and choose a fixed set of independent Hecke characters $\{\lambda_1,\ldots,\lambda_J\}$ modulo $\q$, so $J\leq n_K-1$.  If $J=1$, suppose that $\lambda_1$ has infinite order.  Define for a prime ideal $\p\nmid\q$
\[
\Theta_H(\p)=(\theta_1(\p),\ldots,\theta_J(\p))\in\R^J/\Z^J
\]
by
\[
\lambda_j(\p)=e^{2\pi i \theta_j(\p)},\qquad 1\leq j\leq J.
\]
Let $\mathfrak{I}$ be a narrow ideal class modulo $\q$.  For a collection of closed subintervals $\mathcal{I}=\{I_1,\ldots,I_J\}$ of $[0,\pi]$, we consider the set of primes
\begin{equation}
\label{eqn:general_set}
\mathcal{P}_{\mathcal{I},\mathfrak{I},\mu}=\left\{\textup{$p:$ there exists $\p\subset\mathcal{O}_K$ such that $p=\mathrm{N}\p,\p\in\mathfrak{I},$ and $\theta_H(\p)\in \prod_{j=1}^J I_j$}\right\},
\end{equation}
where $\theta_j(\p)$ is equidistributed in $I_j$ with respect to the probability measure $\mu$ for all $1\leq j\leq J$.  We now define the prime counting function
\[
\pi_{\mathcal{I},\mathfrak{I},\mu}(x;q,a):=\#\{p\leq x : p\equiv a\pmod q, p\in\mathcal{P}_{\mathcal{I},\mathfrak{I},\mu}\}.
\]
For convenience, we let $\delta_{\mathcal{I},\mathfrak{I},\mu}$ denote the density of $\mathcal{P}_{\mathcal{I},\mathfrak{I},\mu}$ within the set of all primes.  This density will depend on the number of real embeddings of $K$, the class number of $K$, and $\varphi(\q)$, which is the integral ideal generalization of Euler's function.  (See the proof of Theorem 3.1 of \cite{Duke} for further discussion.)  The goal of this section is to prove the following theorem.

\begin{theorem}
\label{thm:BV}
Let $E$ denote the Hilbert class field of $K$, which has absolute discriminant $d_E$.  If $0\leq\theta<\frac{1}{9n_K}$ is fixed, then for any fixed $D>0$, we have that
\[
\sideset{}{'}\sum_{\substack{q\leq x^\theta }}\max_{(a,q)=1}\max_{y\leq x}\left|\pi_{\mathcal{I},\mathfrak{I},\mu}(y;q,a)-\delta_{\mathcal{I},\mathfrak{I},\mu}\frac{\pi(y)}{\varphi(q)}\right|\ll\frac{x}{(\log x)^D},
\]
where $\sum'$ denotes summing over moduli $q$ such that $(q,d_E)=1$.
\end{theorem}

Theorem \ref{thm:BV} is a consequence of the following proposition, which we will prove later.

\begin{proposition}
\label{prop:BV}
Fix $D>0$, and define
\[
\Theta_{\mathcal{I},\mathfrak{I},\mu}(x;q,a,k)=\frac{1}{k!}\sum_{\substack{\mathrm{N}\p^m\leq x,~m\geq1  \\ \mathrm{N}\p^m\equiv a\pmod q \\ \mathrm{N}\p\in \mathcal{P}_{\mathcal{I},\mathfrak{I},\mu}}}(\log \mathrm{N}\p)\left(\log\frac{x}{\mathrm{N}\p^m}\right)^k.
\]
If $0\leq\theta<\frac{1}{9n_K}$ is fixed, then
\[
\sideset{}{'}\sum_{\substack{q\leq x^\theta }}\max_{(a,q)=1}\max_{y\leq x}\left|\Theta_{\mathcal{I},\mathfrak{I},\mu}(y;q,a,n_K)-\delta_{\mathcal{I},\mathfrak{I},\mu}\frac{y}{\varphi(q)}\right|\ll\frac{x}{(\log x)^D}.
\]
\end{proposition}

Assuming Proposition \ref{prop:BV}, we prove Theorem \ref{thm:BV}.

\begin{proof}[Proof of Theorem \ref{thm:BV}.]
The theorem follows from Proposition \ref{prop:BV} by a standard application of the mean value theorem.  The arguments are identical to those in Section 1 (and Lemma 1.2 in particular) of \cite{MP} or Section 3 of Gallagher \cite{Gallagher2}.
\end{proof}

\section{Proof of Theorem \ref{thm:BV}}

To prove Proposition \ref{prop:BV}, we first make a series of reductions which reduce the proof to a calculation involving the zero density estimate from Theorem \ref{thm:duke_density}.  We then apply Theorem \ref{thm:duke_density} to finish the proof.  Unless otherwise specified, all implied constants in this section depend in an effectively computable manner on at most $\q$ and $d_E$.

\subsection{Initial reductions}

We begin with the statement of the Erd\"os-Tur\'an inequality for sequences which are equidistributed with respect to a probability measure $\mu$.  This particular version is due to Murty and Sindha \cite{MS}, though the version we use here is slightly weaker.

\begin{lemma}[Theorem 8 of \cite{MS}]
\label{lem:ET}
If a sequence $(a_n)$ of real numbers is equidistributed with respect to a probability measure $\mu$ in the interval $[a,b]$ and $e(t)=e^{2\pi it}$, then for any $x\geq2$ and $T\geq2$,
\[
|\#\{n\leq x:a_n\in [a,b]\}-\mu([a,b])x|\leq \frac{x}{T}+\sum_{1\leq |m|\leq T}\left(\frac{1}{T}+\frac{1}{m}\right)\left|\sum_{n\leq x}e(a_n m)\right|.
\]
\end{lemma}

The proof of Lemma \ref{lem:ET} is easily adapted to accommodate joint distributions of independent equidistributed sequences.  In particular, we have that if $(q,d_E)=1$, $\delta(\xi\lambda^{\m}\otimes\chi)$ is the indicator function for the trivial character, and we define
\[
\Lambda_{\xi\lambda^{\m}\otimes\chi}(\a)=\begin{cases}
\xi\lambda^{\m}\otimes\chi(\p^m)\log\mathrm{N}\p&\mbox{if $\a=\p^m$ for some prime ideal $\p$ and $m\geq1$,}\\
0&\mbox{otherwise},
\end{cases}
\]
then
\begin{align*}
&\left|\Theta_{\mathcal{I},\mathfrak{I},\mu}(x;q,a,k)-\delta_{\mathcal{I},\mathfrak{I},\mu}\frac{x}{\varphi(q)}\right|\\
&\ll \frac{x}{T\varphi(q)}+\frac{1}{\varphi(q)}\sum_{\substack{\chi\bmod q\\ \xi\bmod\q}}\sum_{|\m|\leq T}\left|\frac{1}{k!}\sum_{\mathrm{N}\a\leq x}\Lambda_{\xi\lambda^{\m}\otimes\chi}(\a)\left(\log\frac{x}{\mathrm{N}\a}\right)^k-\delta(\xi\lambda^{\m}\otimes\chi)x\right|.
\end{align*}

From the proof of \cite[Exercise 4.1.6]{Murty}, if $\sigma_0=1+\frac{1}{\log x}$ and $k\geq1$ is an integer, then for $T\geq2$ and $x\geq0$, we have that
\begin{equation}
\label{eqn:mellin}
\left|\frac{1}{2\pi i}\int_{\sigma_0-iT}^{\sigma_0+iT}\frac{x^s}{s^{k+1}}ds-\begin{cases}
\frac{1}{k!}(\log x)^k&\mbox{if $x\geq1$},\\
0&\mbox{if $x\leq 1$}
\end{cases}\right|
\ll \frac{x}{T^{k+1}}.
\end{equation}
Using \eqref{eqn:mellin} and the residue theorem, we have
\[
\frac{1}{k!}\sum_{\mathrm{N}\a\leq x}\Lambda_{\xi\lambda^{\m}\otimes\chi}(\a)\left(\log\mathrm{N}\a\right)\left(\log\frac{x}{\mathrm{N}\a}\right)^k-\delta(\xi\lambda^{\m}\otimes\chi)x\ll \sum_{\substack{0\leq\beta<1 \\ |\gamma|\leq T}}\frac{x^\beta}{|\rho|^k}+O\left(\frac{x(\log x)^2}{T^{k+1}}\right),
\]
where the right-hand sum extends over the nontrivial zeros $\rho=\beta+i\gamma$ of $L(s,(\xi\lambda^{\m}\otimes\chi)')$ and $(\xi\lambda^{\m}\otimes\chi)'$ is the primitive character which induces $\xi\lambda^{\m}\otimes\chi$.  Therefore, we have that for any $2\leq T\leq x$ and any integer $k\geq1$,
\begin{align*}
\Theta_{\mathcal{I},\mathfrak{I},\mu}&(x;q,a,k)-\delta_{\mathcal{I},\mathfrak{I},\mu}\frac{x}{\varphi(q)}\\
&\ll \frac{x}{\varphi(q)T}+\frac{1}{\varphi(q)}\sum_{\substack{\chi\bmod q \\ \xi\bmod\q}}~\sum_{|\m|\leq T}\left(\sum_{\substack{ |\gamma|\leq T \\ 0\leq\beta<1}}\frac{x^{\beta}}{|\rho|^k}+O\left(\frac{x(\log x)^2}{T^{k+1}}\right)\right),
\end{align*}
where the innermost sum extends over all nontrivial zeros $\rho=\beta+i\gamma$ of $L(s,(\xi\lambda^{\m}\otimes\chi)')$, $(\xi\lambda^{\m}\otimes\chi)'$ being the primitive character which induces $\xi\lambda^{\m}\otimes\chi$.  Thus
\begin{align}
\label{eqn:sum_1}
&\sideset{}{'}\sum_{q\leq Q}\max_{(a,q)=1}\max_{y\leq x}\left|\Theta_{\mathcal{I},\mathfrak{I}}(y;q,a,k)-\delta_{\mathcal{I},\mathfrak{I},\mu}\frac{y}{\varphi(q)}\right|\notag\\
&\ll \sideset{}{'}\sum_{q\leq Q}\left|\frac{x}{\varphi(q)T}+\frac{1}{\varphi(q)}\sideset{}{^*}\sum_{\substack{\chi\bmod q \\ \xi\bmod\q}}~\sum_{|\m|\leq T}\left(\sum_{\substack{ |\gamma|\leq T \\ 0\leq\beta<1}}\frac{x^{\beta}}{|\rho|^k}+O\left(\frac{x(\log x)^2}{T^{k+1}}\right)\right)\right|,
\end{align}
where $\sum^*$ denotes summing over primitive characters $\xi\otimes\chi$.  

For a given Hecke $L$-function, the number of zeros $\rho$ with $|\rho|<\frac{1}{4}$ is $\ll \log x$, and by considering the conjugate zero, we deduce that $|\rho|\gg x^{-\epsilon}$ when $q$ is sufficiently large and $0<\epsilon<\frac{1}{4k}$.  Therefore, \eqref{eqn:sum_1} is bounded by
\begin{equation}
\label{eqn:sum_2}
\sideset{}{'}\sum_{q\leq Q}\left|\frac{x}{\varphi(q)T}+\frac{1}{\varphi(q)}\sideset{}{^*}\sum_{\substack{\chi\bmod q \\ \xi\bmod\q}}~\sum_{|\m|\leq T}\left(\sum_{\substack{ |\gamma|\leq T \\ \frac{1}{2}\leq\beta<1}}\frac{x^{\beta}}{|\rho|^k}+O\left(\frac{x(\log x)^2}{T^{k+1}}+\sqrt{x}\right)\right)\right|.
\end{equation}
We now decompose the interval $[1,Q]$ into $O(\log Q)$ dyadic intervals of the form $[2^n,2^{n+1})$.  Thus, using the fact that $\frac{1}{q}\ll\frac{1}{\varphi(q)}\ll\frac{\log\log q}{q}$, \eqref{eqn:sum_2} is bounded by
\begin{align*}
&(\log Q)^2\max_{Q_1\leq Q}\frac{1}{Q_1}\sideset{}{'}\sum_{q\leq Q_1}\left|\frac{x}{T}+\sideset{}{^*}\sum_{\substack{\chi\bmod q \\ \xi\bmod\q}}~\sum_{|\m|\leq T}\left(\sum_{\substack{ |\gamma|\leq T \\ \frac{1}{2}\leq\beta<1}}\frac{x^{\beta}}{|\rho|^k}+O\left(\frac{x(\log x)^2}{T^{k+1}}+\sqrt{x}\right)\right)\right|\\
&\ll (\log Q)^2\max_{Q_1\leq Q}\frac{1}{Q_1}\sideset{}{'}\sum_{q\leq Q_1}~\sideset{}{^*}\sum_{\substack{\chi\bmod q \\ \xi\bmod\q}}~\sum_{|\m|\leq T}\sum_{\substack{ |\gamma|\leq T \\ \frac{1}{2}\leq\beta<1}}\frac{x^{\beta}}{|\rho|^k}+(\log Q)^2\left(\frac{x}{T}+\frac{Qx(\log x)^2}{T^{k+1-n_K}}+T^{n_K}Q\sqrt{x}\right).
\end{align*}
We now embed the above sum into a sum over {\it all} primitive ray class characters $\omega\bmod\a$ with $\mathrm{N}\a\leq Q_1^{n_K}\mathrm{N}\q$.  Specifically, if $\rho=\beta+i\gamma$ is a zero of $L(s,\omega\lambda^{\m})$ with $\omega\bmod\a$ a primitive ray class character, then
\begin{align*}
\frac{1}{Q_1}\sideset{}{'}\sum_{q\leq Q_1}~\sideset{}{^*}\sum_{\substack{\chi\bmod q \\ \xi\bmod\q}}\sum_{|\m|\leq T}\sum_{\substack{|\gamma|\leq T \\ \frac{1}{2}\leq\beta<1}}\frac{x^{\beta}}{|\rho|^k}&\ll\frac{1}{Q_1} \sum_{\mathrm{N}\a\leq Q_1^{n_K}}~\sideset{}{^*}\sum_{\substack{\omega\bmod \a }}~\sum_{|\m|\leq T}\sum_{\substack{|\gamma|\leq T \\ \frac{1}{2}\leq\beta<1}}\frac{x^{\beta}}{|\rho|^k}\\
&\ll\frac{1}{Q_1} \sum_{\mathrm{N}\a\leq Q_1^{n_K}}~\sideset{}{^*}\sum_{\substack{\omega\bmod \a }}~\sum_{|\m|\leq T}\sum_{\substack{|\gamma|\leq T \\ \frac{1}{2}\leq\beta<1}}x^{\beta}\\
&\ll\frac{\log x}{Q_1} \sum_{\mathrm{N}\a\leq Q_1^{n_K}}~\sideset{}{^*}\sum_{\substack{\omega\bmod \a }}~\sum_{|\m|\leq T}\max_{\frac{1}{2}\leq \sigma<1}x^{\sigma}N(\sigma,T,\omega\lambda^{\m})\\
&\ll\frac{\log x}{Q_1} \max_{\frac{1}{2}\leq \sigma<1}x^{\sigma}N(\sigma,Q_1^{n_K}\mathrm{N}\q,T). 
\end{align*}

Collecting the above estimates, we have that if $Q\ll T\ll x$ and $k=n_K$, then
\begin{align}
\label{eqn:main_ineq}
&\sideset{}{'}\sum_{q\leq Q}\max_{(a,q)=1}\max_{y\leq x}\left|\Theta_{\mathcal{I},\mathfrak{I}}(y;q,a,n_K)-\delta_{\mathcal{I},\mathfrak{I},\mu}\frac{y}{\varphi(q)}\right|\notag\\
&\ll (\log x)^3\left(\max_{Q_1\leq Q}\max_{\frac{1}{2}\leq\sigma<1}\frac{x^{\sigma}}{Q_1}N(\sigma,Q_1^{n_K}\mathrm{N}\q,T)+\frac{Qx(\log x)^2}{T}+Q T^{n_K}\sqrt{x}\right).
\end{align}
Thus we have reduced the proof of Proposition \ref{prop:BV} to a calculation involving the density estimate in Theorem \ref{thm:duke_density}.

\subsection{Finishing the proof}

Choose $\epsilon$ as before, and set
\[
Q=x^{\frac{1-\epsilon}{9n_K}}(\log x)^{-\frac{D+2}{3}}\qquad\textup{and}\qquad T=x^{\frac{1-\epsilon}{9n_K}}(\log x)^{\frac{2(D+2)}{3}}.
\]
By the zero-free regions for Hecke $L$-functions proven Coleman \cite{Coleman} and the fact that we restrict $q$ to be coprime to $d_E$, there exists a constant $c$ (depending at most on $\q$ and $K$) such that if $|\m|\leq T$, the modulus of $\omega$ is at most $O(Q_1^{n_K})$, and
\begin{equation}
\label{ZFR-vino}
1-\eta(Q_1,x)<\sigma\leq 1,\qquad \eta(Q_1,x):=\frac{c}{\max\{\log Q_1,(\log x)^{\frac{3}{4}}\}},
\end{equation}
then $N(\sigma,Q_1^{n_E},T)$ is either 0 or 1.  If $N(\sigma,Q_1^{n_E},T)=1$, then the zero $\beta_1$ which is counted is a Siegel zero.  As in \cite[Section 2]{MP}, a field-uniform version of Siegel's theorem for Hecke $L$-functions implies that $x^{\beta_1-1}\ll(\log x)^{-D-3}$ with an ineffective implied constant. 

Note that $(Q^2 T)^{3n_K}=x^{1-\epsilon}$.  Since $2-\sigma\geq1$ for all $\frac{1}{2}\leq \sigma\leq 1$, we have that
\begin{align*}
\max_{\frac{1}{2}\leq\sigma\leq1}N(\sigma,Q_1^{n_K},T)x^{\sigma}&\ll (\log x)^{A+1}\max_{\frac{1}{2}\leq\sigma\leq1-\sigma(Q_1,x)}(Q_1^2 T)^{\frac{3n_K(1-\sigma)}{2-\sigma}}x^{\sigma}\\
&\ll x(\log x)^{A+1}\max_{\frac{1}{2}\leq\sigma\leq1-\sigma(Q_1,x)}x^{\epsilon(\sigma-1)}\\
&\ll x^{1-\epsilon\sigma(Q_1,x)}(\log x)^{A+1}.
\end{align*}
Furthermore,
\begin{equation}
\label{eqn:many_logs}
x^{-\epsilon\sigma(Q_1,x)}(\log x)^{A+1}\ll\begin{cases}
(\log x)^{-D-5}&\mbox{if $Q_1\leq \exp((\log x)^{\frac{1}{4}})$,}\\
1&\mbox{if $\exp((\log x)^{\frac{1}{4}})<Q_1\leq Q$}.
\end{cases}
\end{equation}
Collecting our previous estimates, recalling our choices of $Q$ and $T$, and using \eqref{eqn:many_logs}, we see that \eqref{eqn:main_ineq} is bounded by
\begin{align*}
\frac{Qx(\log x)^2}{T}+QT^{n_K}\sqrt{x}+\frac{x^{1-\epsilon\sigma(Q_1,x)}(\log x)^{A+2}}{Q_1}\ll \frac{x}{(\log x)^D},
\end{align*}
as desired.  This completes the proof of Proposition \ref{prop:BV}.


\section{Proofs of Theorems \ref{mainthm:angles1} and \ref{mainthm:angles2}}


We will use Theorem \ref{thm:BV} to prove a very general result on bounded gaps between primes in sets of the form \eqref{eqn:general_set}, from which we will deduce Theorems \ref{mainthm:angles1} and \ref{mainthm:angles2}.  Given a set of integers $\mathfrak{A}$, a set of primes $\mathfrak{P}\subset\mathfrak{A}$, and a linear form $L(n)=n+h$, define
\begin{equation*}
\begin{aligned}[c]
\mathfrak{A}(x)&=\{n\in\mathfrak{A}:x< n\leq 2x\},\\
L(\mathfrak{A})&=\{L(n):n\in\mathfrak{A}\},\\
\mathfrak{P}_{L,\mathfrak{A}}(x,y)&=L(\mathfrak{A}(x))\cap\mathfrak{P},
\end{aligned}
\begin{aligned}[c]
\mathfrak{A}(x;q,a)&=\{n\in\mathfrak{A}(x):n\equiv a~(\mathrm{mod}~q)\},\\
\varphi_{L}(q)&=\varphi(h q)/\varphi(h),\\
\mathfrak{P}_{L,\mathfrak{A}}(x;q,a)&=L(\mathfrak{A}(x;q,a))\cap\mathfrak{P}.
\end{aligned}
\end{equation*}
We consider the 6-tuple $(\mathfrak{A},\mathcal{L}_k,\mathfrak{P},B,x,\theta)$, where $\mathcal{H}_k$ is admissible, $\mathcal{L}_k=\{L_i(n)=n+h_i:h_i\in\mathcal{H}_k\}$, $B\in\N$ is constant, $x$ is a large real number, and $0\leq\theta<1$.  We present a very general hypothesis that Maynard states in Section 2 of \cite{maynard2}.

\begin{hypothesis}
\label{hyp}
With the above notation, consider the 6-tuple $(\mathfrak{A},\mathcal{H}_k,\mathfrak{P},B,x,\theta)$.
\begin{enumerate}
\item We have
\[
\sum_{q\leq x^\theta}\max_a\left|\#\mathfrak{A}(x;q,a)-\frac{\#\mathfrak{A}(x)}{q}\right|\ll \frac{\#\mathfrak{A}(x)}{(\log x)^{100k^2}}.
\]
\item For any $L\in \mathcal{H}_k$, we have
\[
\sum_{\substack{q\leq x^{\theta} \\ (q,B)=1}}\max_{(L(a),q)=1}\left|\#\mathfrak{P}_{L,\mathfrak{A}}(x;q,a)-\frac{\#\mathfrak{P}_{L,\mathfrak{A}}(x)}{\varphi_L(q)}\right|\ll \frac{\#\mathfrak{P}_{L,\mathfrak{A}}(x)}{(\log x)^{100k^2}}.
\]
\item For any $q\leq x^\theta$, we have $\#\mathfrak{A}(x;q,a)\ll \#\mathfrak{A}(x)/q.$
\end{enumerate}
\end{hypothesis}
\noindent
For $(\mathfrak{A},\mathcal{H}_k,\mathfrak{P},B,x,\theta)$ satisfying Hypothesis \ref{hyp}, Maynard proves the following in \cite{maynard2}.

\begin{theorem}
\label{big-maynard-thm}
Let $\alpha>0$ and $0\leq\theta<1$.  There is a constant $C$ depending only on $\theta$ and $\alpha$ so that the following holds.  Let $(\mathfrak{A},\mathcal{H}_k,\mathfrak{P},B,x,\theta)$ satisfy Hypothesis \ref{hyp}.  Assume that $C\leq k\leq (\log x)^\alpha$ and $h_i\leq x^\alpha$ for all $1\leq i\leq k$.  If $\delta>(\log k)^{-1}$ is such that
\[
\frac{1}{k}\frac{\varphi(B)}{B}\sum_{L_i\in\mathcal{H}_k}\#\mathfrak{P}_{L_i,\mathfrak{A}}(x)\geq\delta\frac{\#\mathfrak{A}(x)}{\log x},
\]
then
\[
\#\{n\in\mathfrak{A}(x):\#(\mathcal{H}_k(n)\cap\mathfrak{P})\geq C^{-1}\delta\log k\}\gg\frac{\#\mathfrak{A}(x)}{(\log x)^k \exp(Ck)}.
\]
\end{theorem}

Using Theorem \ref{big-maynard-thm}, we prove the following result.

\begin{theorem}
\label{thm:hecke_ED_gaps}
Let $\mathcal{P}_{\mathcal{I},\mathfrak{I},\mu}$ be a set of the form \eqref{eqn:general_set}.  There exists a constant $C_{\mathcal{I},\mathfrak{I},\mu}>0$ such that if $k>C_{\mathcal{I},\mathfrak{I},\mu}$ and $\mathcal{H}_k=\{h_1,\ldots,h_k\}$ is admissible, then
\[
\#\{n\leq x : \#(\{n+h_1,\ldots,n+h_k\}\cap\mathcal{P}_{\mathcal{I},\mathfrak{I},\mu})\geq C_{\mathcal{I},\mathfrak{I},\mu}^{-1}\log k\}\gg\frac{x}{(\log x)^k}.
\]
\end{theorem}

\begin{proof}
The proof is essentially the same as Theorem 3.5 in \cite{maynard2}.  Let $\theta$ be as in Theorem \ref{thm:BV}.  Let $\mathfrak{A}=\N$, $\mathfrak{P}=\mathcal{P}_{\mathcal{I},\mathfrak{I},\mu}$, and $B=d_E$.  Parts (i) and (iii) of Hypothesis \ref{hyp} are trivial to check for the 6-tuple $(\N,\mathcal{H}_{k},\mathcal{P}_{\mathcal{I},\mathfrak{I},\mu},d_E,x,\theta/2)$.  By Theorem \ref{thm:BV} partial summation, all of Hypothesis \ref{hyp} holds when $D$ and $x$ are sufficiently large in terms of $k$ and $\theta$.

Given a suitable constant $C_{\mathcal{I},\mathfrak{I},\mu}>0$ (computed as in \cite{maynard, JT}), we let $k\geq C_{\mathcal{I},\mathfrak{I},\mu}$.  For our choice of $\mathfrak{A}$ and $\mathfrak{P}$, we have the inequality
\[
\frac{1}{k}\frac{\varphi(B)}{B}\sum_{L_i\in\mathcal{H}_k}\#\mathfrak{P}_{L_i,\mathfrak{A}}(x)\geq(1+o(1))\frac{\varphi(B)}{B}\delta_{\mathcal{I},\mathfrak{I},\mu}\frac{\#\mathfrak{A}(x)}{\log x}
\]
for all sufficiently large $x$, where the implied constant in $1+o(1)$ depends only on the Hilbert class field of $K$.  Theorem \ref{thm:hecke_ED_gaps} now follows directly from Theorem \ref{big-maynard-thm}.
\end{proof}

To prove Theorems \ref{mainthm:angles1} and \ref{mainthm:angles2}, it now suffices to show that the sets of primes considered in the respective theorems are both of the form \eqref{eqn:general_set} for certain sets of independent Hecke characters.  For Theorem \ref{mainthm:angles1}, this is accomplished in the proof of \cite[Theorem 3.2]{Duke}; for Theorem \ref{mainthm:angles2}, this is accomplished in the proof of \cite[Theorem  3.3]{Duke}.

\bibliographystyle{abbrv}
\bibliography{JAThorner_bounded_gaps_Hecke}
\end{document}